\documentclass[11pt]{article}
\usepackage{longtable}
\usepackage{amssymb}
\usepackage{amsmath}
\usepackage{amsthm}
\usepackage{enumerate}
\usepackage{color}
\newtheoremstyle{plain}%
    {8pt plus2pt minus4pt}%
    {8pt plus2pt minus4pt}%
    {\itshape}%
    {}%
    {\bfseries\scshape}%
    {}%
    {6pt}% Space after theorem head
    {}%

\newtheoremstyle{remark}%
    {8pt plus2pt minus4pt}%
    {8pt plus2pt minus4pt}%
    {\upshape}% Body font
    {}%
    {\bfseries\scshape}%
    {}%
    {6pt}% Space after theorem head
    {}%

\theoremstyle{plain}
\newtheorem{theorem}{Theorem}%[section]

\newtheorem{conjecture}[theorem]{Conjecture}

\newtheorem{lemma}[theorem]{Lemma}

\newtheorem{definition}[theorem]{Definition}
\newtheorem{claim}[theorem]{Claim}

\theoremstyle{remark}

\providecommand{\bysame}{\leavevmode\hbox to3em{\hrulefill}\thinspace}

\def\le{\leqslant}
\def\ge{\geqslant}

\def\<{\langle}
\def\>{\rangle}

\def\ex{\textup{ex}}
\def\Forb{\textup{Forb}}
\title{The number of triple systems without even cycles}

\author{
{\Large{Dhruv Mubayi}}\thanks{
\footnotesize {Department of Mathematics, Statistics, and Computer Science, University of Illinois, Chicago, IL 60607, \texttt{Email:mubayi@uic.edu}}. Research partially supported by NSF grant DMS-1300138.}
\and
{\Large{Lujia Wang}}\thanks{
\footnotesize {Department of Mathematics, Statistics, and Computer Science, University of Illinois, Chicago, IL 60607, \texttt{Email:lwang203@uic.edu}}. Research partially supported by NSF grant DMS-1300138.}
}
\date{\today}

\begin{document}
\maketitle
\begin{abstract}
For $k \ge 4$, a loose $k$-cycle $C_k$ is a hypergraph with distinct edges $e_1, e_2, \ldots, e_k$ such that consecutive edges (modulo $k$) intersect in exactly one vertex and all other pairs of edges are disjoint. Our main result is that for every even integer $k \ge 4$, there exists  $c>0$  such that the number of triple systems with vertex set $[n]$ containing no $C_{k}$ is at most $2^{cn^2}$.

An easy construction shows that the exponent is sharp in order of magnitude. This may be viewed as a hypergraph extension of the work of Morris and Saxton, who proved the analogous result for graphs which was a longstanding problem. For $r$-uniform hypergraphs with $r>3$,  we improve the trivial upper bound but fall short of obtaining the order of magnitude in the exponent, which we conjecture is $n^{r-1}$.

Our proof method is different than that used for most recent results of a similar flavor about enumerating discrete structures, since it does not use hypergraph containers. One novel ingredient is the use of some (new) quantitative estimates for an asymmetric version of the bipartite canonical Ramsey theorem.

\end{abstract}

\section{Introduction}

An important theme in combinatorics is the enumeration of discrete structures that have certain properties. Within extremal combinatorics, one of the first influential results  of this type is the Erd\H os-Kleitman-Rothschild theorem~\cite{EKR76}, which implies that the number of triangle-free graphs with vertex set $[n]$ is $2^{n^2/4+o(n^2)}$. This has resulted in a great deal of work on problems about counting the number of graphs with other forbidden subgraphs~\cite{BBS04,BBS09,BBS11,BSmm,BSst,EFR86,HPS93,LR84,PS92} as well as similar question  for other discrete structures~\cite{BM11,BM12,BG96,BGP99,KR75,OKTZ14,PS09,R73,S73}. In extremal graph theory, these results show that such problems are closely related to the corresponding extremal problems. More precisely, the \emph{Tur\'{a}n} problem asks for the maximum number of edges in a (hyper)graph that does not contain a specific subgraph. For a given $r$-uniform hypergraph (henceforth $r$-graph) $F$,  let $\ex_r(n,F)$ be  the maximum number of edges among all $r$-graphs $G$ on $n$ vertices that contain no copy of $F$ as a (not necessarily induced) subgraph. Henceforth we will call $G$ an  $F$-free $r$-graph. Write $\Forb_r(n,F)$ for the set of $F$-free $r$-graphs with vertex set $[n]$. Since each subgraph of an $F$-free $r$-graph is also $F$-free, we trivially obtain $|\Forb_r(n,F)| \ge 2^{\ex_r(n,F)}$ by taking subgraphs of an 
$F$-free $r$-graph on $[n]$ with the  maximum number of edges. On the other hand for fixed $r$ and  $F$,
$$|\Forb_r(n,F)|\le \sum_{i \le \ex_r(n,F)} \binom{\binom{n}{r}}{i} =  2^{O(\ex_r(n,F)\log n)},$$
so the issue at hand is the factor $\log n$ in the exponent.  The work of 
Erd\H os-Kleitman-Rothschild~\cite{EKR76} and Erd\H os-Frankl-R\"odl~\cite{EFR86} for graphs, and Nagle-R\" odl-Schacht~\cite{NRS06} for hypergraphs (see also~\cite{NR01} for the case $r=3$) improves the upper bound above to obtain
$$|\Forb_r(n,F)|=2^{\ex_r(n,F)+o(n^r)}.$$
Although much work has been done to improve the exponent above (see \cite{ABBM11,BBS04,BBS09,BBS11,HPS93, KPR87, PS92} for graphs and~\cite{BM11,BM12,CG16, PS09,BMS15,ST15} for hypergraphs), this is a somewhat satisfactory state of affairs when $\ex_r(n,F)=\Omega(n^r)$ or $F$ is not $r$-partite.

In the case of $r$-partite $r$-graphs, the corresponding questions appear to be more challenging since the tools used to address the case $\ex_r(n,F)= \Omega(n^r)$ like the regularity lemma are not applicable. The major open problem here when $r=2$ is to prove that 
$$|\Forb_r(n,F)|=2^{O(\ex_r(n, F))}.$$
 The two cases that have received the most attention are for $r=2$ (graphs) and $F=C_{2l}$ or $F=K_{s,t}$. Classical results of Bondy-Simonovits~\cite{BS74} and Kov\'ari-S\'os-Tur\'an~\cite{KST54} yield 
$\ex_2(n, C_{2l})=O(n^{1+1/l})$ and $\ex_2(n, K_{s,t})=O(n^{2-1/s})$ for $2\le s\le t$, respectively. Although it is widely believed that these upper bounds  give the correct order of magnitude, this is not known in all cases. Hence the enumerative results sought in these two cases were 
$$|\Forb_2(n, C_{2l})|=2^{O(n^{1+1/l})} \qquad \text{\rm and} \qquad |\Forb_2(n, K_{s,t})|=2^{O(n^{2-1/s})}.$$ 

In 1982, Kleitman and Winston~\cite{KW82} proved that $|\Forb_2(n,C_4)|= 2^{O(n^{3/2})}$ which initiated a 30-year effort on searching for generalizations of the result to complete bipartite graphs and even cycles. Kleitman and Wilson~\cite{KW96} proved similar results for $C_6$ and $C_8$ in 1996 by reducing to the $C_4$ case. Finally, Morris and Saxton~\cite{MS13} recently  proved that $|\Forb_2(n,C_{2l})|= 2^{O(n^{1+1/l})}$ and Balogh and Samotij~\cite{BSmm,BSst}  proved that $|\Forb_2(n,K_{s,t})|= 2^{O(n^{2-1/s})}$ for $2\le s\le t$. Both these results used the  hypergraph container method (developed independently by Saxton and Thomason~\cite{ST15}, and by Balogh-Morris-Samotij~\cite{BMS15}) which has proved to be a very powerful technique in extremal combinatorics.
For example, \cite{BMS15} and \cite{ST15} reproved $|\Forb_r(n,F)|= 2^{\ex_r(n,F)+o(n^r)}$ using containers.

There are very few results in this area when $r>2$ and $\ex_r(n, F)=o(n^r)$. The only cases solved so far are when $F$ consists of just two edges that  intersect in at least $t$ vertices~\cite{BDDLS}, or when $F$ consists of three edges such that the union of the first two is equal to the third~\cite{BW16} (see also~\cite{BBM, BBR, DT, DK} for some related results). These are natural points to begin these investigations since the corresponding extremal problems have been studied deeply. 

Recently, Kostochka, the first author and Verstra\"ete~\cite{KMV15,KMV15II,KMV15III}, and independently, F\"uredi and Jiang~\cite{FJ14} (see also \cite{FJS14}) determined the  Tur\'{a}n number for several other families of $r$-graphs including  paths, cycles, trees, and expansions of graphs. These hypergraph extremal problems have proved to be quite difficult, and include some longstanding conjectures. 
 Guided and motivated  by these recent developments on the extremal number of hypergraphs, we consider the corresponding enumeration problems focusing on the case of cycles.
 
\begin{definition} For each integer $k\ge 3$, a $k$-cycle $C_k$  is a hypergraph with distinct edges $e_1,  \ldots, e_k$ and distinct vertices $v_1, \ldots, v_k$ such that 
$e_i\cap e_{i+1} =\{v_i\}$ for all $1\le i \le k-1 $, $e_1 \cap e_k=\{v_k\}$ and $e_i \cap e_j=\emptyset$ for all other pairs $i,j$. 
\end{definition}

Sometimes we refer to $C_k$ as a \emph{loose} or \emph{linear} cycle. To simplify notation, we will omit the parameter $r$ when the cycle $C_k$ is a subgraph of an $r$-graph.  
 
Since $\ex_r(n, C_{k}) =O(n^{r-1})$, we obtain the upper bound 
$$|\Forb_r(n, C_{k})| = 2^{O(n^{r-1}\log n)}$$
when $r$ and $k$ are fixed and $n \rightarrow \infty$.
Our main result is the following theorem, which improves this upper bound and generalizes the Morris-Saxton theorem~\cite{MS13} to 3-graphs. 
\medskip

\begin{theorem}\label{Main} {\bf (Main Result)}
For integers $ r, k\ge 3$, there exists $c=c(r,k)$, such that
\begin{align*}
|\Forb_r(n, C_{k})|< 
\begin{cases}
2^{c\,n^{2}} &\text{ if } r=3 \text{ and $k$ is even,}\\
2^{c \, n^{r-1}(\log n)^{(r-3)/(r-2)}}&\text{ if }r>3.
\end{cases}
\end{align*}
\end{theorem}
Since trivially $\ex_r(n, C_{k})=\Omega(n^{r-1})$ for all $r \ge 3$, we obtain $|\Forb_{3}(n,C_{k})|=2^{\Theta(n^2)}$ when $k$ is even. We conjecture that a similar result holds for $r>3$ and cycles of odd length.
\medskip

\begin{conjecture} \label{conj}
For fixed $r\ge 3$ and $k \ge 3$ we have $|\Forb_{r}(n,C_{k})|=2^{\Theta(n^{r-1})}$.
\end{conjecture}

Almost all recent developments in this area have relied on the method of hypergraph containers that we mentioned above. What is perhaps  surprising about the current work is that the proofs do not use hypergraph containers. Instead, our methods employ old and new tools in extremal (hyper)graph theory. The old tools include the extremal numbers for cycles modulo $h$ and results about decomposing complete $r$-graphs into $r$-partite ones, and the new tools include the analysis of the shadow for extremal hypergraph problems and quantitative estimates for the bipartite canonical Ramsey problem.

\subsection{Definitions and notations}

Throughout this paper, we let $[n]$ denote the set $\{1,2,\ldots,n\}$.  Write $\binom{X}{r}= \{S\subset X:|S|=r\}$ and $\binom{X}{\le r}= \{S\subset X:|S|\le r\}$. For $X \subset[n]$, an \emph{$r$-uniform hypergraph} or \emph{$r$-graph} $H$ on vertex set $X$ is a collection of $r$-element subsets of $X$, i.e. $H\subset\binom{X}{r}$. The vertex set $X$ is denoted by $V(H)$. The $r$-sets contained in $H$ are  \emph{edges}. The \emph{size} of  $H$ is $|H|$. Given $S\subset V(H)$, the \emph{neighborhood} $N_{H}(S)$ of $S$ is the set of all $T\subset V(H)\setminus S$  such that $S\cup T \in H$. The \emph{codegree} of $S$ is $d_{H}(S)=|N_H(S)|$. When the underlying hypergraph is clear from context, we may omit the subscripts in these definitions and write $N(S)$ and $d(S)$ for simplicity. The \emph{sub-edges} of $H$ are the $(r-1)$-subsets of $[n]$ with positive codegree  in $H$. The set of all sub-edges of $H$ is called the \emph{shadow} of $H$, and is denoted  $\partial H$.

An \emph{$r$-partite $r$-graph} $H$ is an $r$-graph with vertex set $\bigsqcup_{i=1}^rV_i$ (the $V_i$s are pairwise disjoint), and every $e\in H$ satisfies $|e\cap V_i|=1$ for all $i\in [r]$. When all such edges $e$ are present,  $H$ is called a \emph{complete $r$-partite $r$-graph}. When $|V_i|=s$ for all $i\in [r]$, a complete $r$-partite $r$-graph $H$ is said to be \emph{balanced}, and denoted $K_{s:r}$.

For each integer $k\ge 1$, a \emph{(loose, or linear) path of length $k$} denoted by $P_{k}$, is a collection of $k$ edges $e_1, e_2,\ldots, e_k$ such that $|e_i\cap e_j|=1$ if $i=j+1$, and $e_i\cap e_j=\emptyset$ otherwise. 

We will often omit floors and ceilings in our calculations for ease of notation and all $\log$s will have base 2.
\section{Proof of the main result}

The proof of Theorem~\ref{Main} proceeds by counting  edge-colored $(r-1)$-graphs with certain restrictions; the details differ quite substantially for the cases $r=3$ and $r>3$. In this section we state the main technical statement (Theorem~\ref{ColorKs:r-1}) about these edge-colorings that will be needed, as well as some other tools, and then prove Theorem~\ref{Main} using these results.

\subsection{Main technical statement}

Given an $(r-1)$-graph $G$ with $V(G)\subset[n]$, a \emph{coloring function} is a function $\chi: G\rightarrow [n]$ such that  $\chi(e)=z_e\in [n]\setminus e$ for every $e \in G$. We call $z_e$ the \emph{color} of $e$. The vector of colors $N_G=(z_e)_{e\in G}$ is called an \emph{edge-coloring} of $G$. The pair $(G,N_G)$ is an \emph{edge-colored} $(r-1)$-graph.  A \emph{color class} is the set of all edges that receive the same color.

Given $G$, each edge-coloring $N_G$ defines an $r$-graph $H(N_G)=\{e\cup \{z_e\}: e\in G\}$, called the \emph{extension of $G$ by $N_G$}. When there is only one coloring that has been defined, we also use the notation $G^*=H(N_G)$ for the extension. Observe that any subgraph $G' \subset G$ also admits an extension by $N_G$, namely, $G'^*=\{e\cup \{z_e\}:e\in G'\} \subset G^*$. If $G' \subset G$ and $\chi|_{G'}$ is one-to-one, then $G'$ is called \emph{rainbow colored}. If a rainbow colored $G'$ further satisfies that $z_e\notin V(G')$ for all $e\in G'$, then $G'$ is said to be \emph{strongly rainbow colored}. Note that  a strongly rainbow colored graph $C_{k} \subset G'$ gives rise to 3-graph $C_{k}$ in $G'^* \subset G^*$.

\begin{definition}
Given $r\ge 3, k\ge 3, s\ge 1$, let $f_{r}(n,k,s)$ be the number of edge-colored balanced complete $(r-1)$-partite $(r-1)$-graphs $G=K_{s:r-1}$ with $V(G)\subset[n]$, whose extension $G^*$ is $C_{k}$-free.
\end{definition}

The function $f_r(n,k,s)$ allows us to encode $r$-graphs, and our main technical theorem gives an upper bound for this function. 

\begin{theorem}\label{ColorKs:r-1}
Given $r\ge 3, k\ge 3$ there exist $D=D_{k/2}$, $c_2=c_2(r,k)$, such that
\begin{align*}
f_{r}(n,k,s)\le
    \begin{cases}
      2^{(5/2)ks\log n+4s^2\log D} & \text{if}\ r=3,  k\text{ is even}, \\
      2^{(c_2+2r)s^{r-2}\log n+s^{r-1}(\log(c_2+r)+(r-2)\log s)} & \text{if}\ r>3.
    \end{cases}
\end{align*}
\end{theorem}
For $r$ and $k$ fixed, the bounds above can be written as
\begin{align*}
f_{r}(n,k,s)=
    \begin{cases}
      2^{O(s\log n + s^2)} & \text{if}\ r=3,  k\text{ is even,} \\
      2^{O(s^{r-2}\log n+s^{r-1}\log s)} & \text{if}\ r>3.
    \end{cases}
\end{align*}
Note that the trivial upper bound is $f_r(n,k,s)\le n^{(r-1)s+s^{r-1}}\sim 2^{s^{r-1}\log n}$ (first choose $(r-1)s$ vertices, then color each of its $s^{r-1}$ edges) so Theorem~\ref{ColorKs:r-1} is nontrivial only  if $s=o(n)$. The proof of Theorem~\ref{ColorKs:r-1} will be given in Sections 3--6.
 
%Let us provide some intuition for the quantities that appear in Theorem~\ref{ColorKs:r-1}. When applying Theorem~\ref{ColorKs:r-1} in our proof, $r$ and $k$ are fixed and $s$ is a  polynomial in $\log n$.  As noted above, the trivial upper bound for $f_r(n,k,s)$ is $n^a=2^{a \log n}$ where $a=|K_{s:r-1}|=s^{r-1}$. Let us suppose for a moment that $r=3$ and one could show that the set of colors of each edge lay in a  subset of $[n]$ of size $O(s)$. Then the upper bound can be improved to roughly $s^{O(a)}=2^{O(s^2\log s)}$, which is quite a bit better than the trivial bound of $2^{O(s^2 \log n)}$ but not quite the bound $2^{O(s\log n+s^2)}$ given by Theorem~\ref{ColorKs:r-1}. Suppose one could further argue that for most edges, there are at most a constant number of choices for the color of the edge (instead of $O(s)$ choices). Then we obtain the bound in Theorem~\ref{ColorKs:r-1} (when $r=3$ and $k$ is even). Our main work is in showing that the $C_{k}$-free assumption allows us to make these sort of quantitative statements precise.  

 \subsection{Decomposing $r$-graphs into balanced complete $r$-partite $r$-graphs}
 
Chung-Erd\H{o}s-Spencer~\cite{CES83} and Bublitz~\cite{B86} proved that the complete graph $K_n$ can be decomposed into balanced complete bipartite graphs such that the sum of the sizes of the vertex sets in these bipartite graphs is at most $O(n^2/\log n)$. See also~\cite{T84,MT10} for some generalizations and algorithmic consequences. We need the following generalization of this result to $r$-graphs.
\begin{theorem}\label{HypergraphDecomposition}
Let $n\ge r\ge 2$. There exists a constant $c'_1=c'_1(r)$, such that any $n$-vertex $r$-graph $H$ can be decomposed into balanced complete $r$-partite $r$-graphs $K_{s_i:r}, i=1,\ldots,m$, with $s_i\le (\log n)^{1/(r-1)}$ and $\sum_{i=1}^ms_{i}^{r-1}\le c'_1n^{r}/(\log n)^{1/(r-1)}$.
\end{theorem}
\begin{proof}
An old result of Erd\H os~\cite{E64} states that for any integers $r,s\ge 2$ and $n\ge rs$, we have
$$\ex_r(n,K_{s:r})<n^{r-1/s^{r-1}}.$$
Note that for $r=2$, this was proved much earlier by K\H{o}v\'{a}ri-S\'{o}s-Tur\'{a}n~\cite{KST54}.

We first assume that $n\ge 2r$. Taking the derivative, one can show that for each $r\ge 2$, $n\longmapsto n/r-(\log n)^{1/(r-1)}$ is an increasing function in $n$, hence its minimum is achieved at $n=2r$. So for all $n\ge 2r$, we have
$$\frac{n}{r}-(\log n)^{1/(r-1)}\ge \frac{2r}{r}-(\log 2r)^{1/(r-1)}=2-(\log 2r)^{1/(r-1)}\ge 0.$$
Thus, for any $s\le (\log n)^{1/(r-1)}\le n/r$, the Tur\'an number for $K_{s:r}$ is 
$$\ex_r(n,K_{s:r})<n^{r-1/s^{r-1}}.$$
Next, we give an algorithm of decomposing $H$ into $K_{s:r}$s. Let $H_1=H$. For $i\ge 1$, repeatedly find a $K_{s_i:r}\subset H_i$ with maximum $s_i$ subject to $s_i\le (\log n)^{1/(r-1)}$ and delete it from $H_i$ to form $H_{i+1}$. The loop terminates at step $i$ if $|H_i|\le n^r/(\log n)^{1/(r-1)}$. Then let the remaining graph be decomposed into single edges ($K_{1:r}$s).

By the algorithm, the vertex size of each $K_{s_i:r}$ satisfies that $s_i\le  (\log n)^{1/(r-1)}, 1\le i\le m$ automatically. So we are left to show the upper bound for $\sum_{i=1}^m s_i^{r-1}$.

We divide the iterations of the above algorithm into phases, where the $k$th phase consists of those iterations where the number of edges in the input $r$-graph of the algorithm lies in the interval $(n^r/(k+1),n^r/k]$. In other words, in phase $k$, each $K_{s_i:r}$ to be found is in an $r$-graph $H_i$ with $|H_i|>n^r/(k+1)$. Define $s(k)= \left({\log n/\log (k+1)}\right)^{1/(r-1)}$. Then it is easy to see $s(k)\le (\log n)^{1/(r-1)}\le n/r$. So by Erd\H os' result
$$\frac{n^r}{k+1}=n^{r-{1/s(k)^{r-1}}}>\ex(n,K_{s(k):r}).$$
Hence, $K_{s(k):r}\subset H_i$. So in phase $k$, the minimum $s_i$ of a $K_{s_i:r}$ we are able to find has the lower bound
\begin{align*}
s_i\ge s(k)= \left(\frac{\log n}{\log (k+1)}\right)^{1/(r-1)}.
\end{align*}
Now notice that
\begin{align*}
\sum_{i=1}^ms_i^{r-1}=\sum_{i=1}^ms_i^r\cdot\frac{1}{s_i}.
\end{align*}
Dividing up the terms in the summation according to phases, we observe that this is a sum of the number of edges deleted in the $k$th phase times a weight of $1/s_i$ for each edge. Also notice that there are at most $n^r/(\log n)^{1/(r-1)}$ single edges, we have
\begin{align*}
\sum_{i=1}^ms_i^{r-1}&\le \frac{n^r}{(\log n)^{1/(r-1)}}+\sum_{k=1}^\infty n^r\left(\frac{1}{k}-\frac{1}{k+1}\right)\left(\frac{\log (k+1)}{\log n}\right)^{1/(r-1)}\\
&= \frac{n^r}{(\log n)^{1/(r-1)}}+ \sum_{k=1}^\infty\frac{n^r(\log (k+1))^{1/(r-1)}}{k(k+1)(\log n)^{1/(r-1)}}\\
&= \frac{c'_1n^r}{(\log n)^{1/(r-1)}},
\end{align*}
where 
$$c'_1=1+\sum_{k=1}^\infty\frac{(\log(k+1))^{1/(r-1)}}{k(k+1)}.$$

Finally, for $n<2r$, we just let $H$ be decomposed into $K_{1:r}$s. Then we obtain the following bound for $\sum_{i=1}^ms_i^{r-1}$.
$$\sum_{i=1}^ms_i^{r-1}= m\le \binom{2r}{r}\le\frac{c'_1n^r}{(\log n)^{1/(r-1)}}$$
with appropriately chosen $c'_1=c'_1(r)$. This completes the proof. \end{proof}

\subsection{A corollary of Theorem~\ref{ColorKs:r-1}}

Theorem~\ref{ColorKs:r-1} is about the number of ways to edge-color complete $(r-1)$-partite $(r-1)$-graphs with parts of size $s$ and vertex set in $[n]$. In this section, we use Theorems~\ref{ColorKs:r-1} and \ref{HypergraphDecomposition} to prove a related statement where we do not require the $(r-1)$-partite condition and the restriction to $s$ vertices.

\begin{definition}
For $r \ge 3$ and $k \ge 3$, let $g_r(n,k)$ be the number of edge-colored $(r-1)$-graphs $G$ with $V(G)\subset [n]$ such that the extension $G^*$ is $C_{k}$-free.
\end{definition}
\medskip

\begin{lemma}\label{ColorHyperG}
Let $r\ge 3, k\ge 3$, and $n$ be large enough. Then there exist $c_1=c_1(r)$, $c_2=c_2(r,k)$ and $D=D_{k/2}$, such that
\begin{align*}
g_r(n,k)\le
\begin{cases}
      2^{(3kc_1+4\log D)n^2} & \text{if}\ r=3, k\text{ is even,} \\
      2^{2(c_2+2r)c_1n^{r-1}(\log n)^{(r-3)/(r-2)}} & \text{if}\ r>3.
    \end{cases}
\end{align*}
Note that if $r$ and $k$ are fixed, for both cases we have
\begin{align*}
g_r(n,k)=2^{O(n^{r-1}(\log n)^{(r-3)/(r-2)})}.
\end{align*}
\end{lemma}
\begin{proof}
Given any $(r-1)$-graph $G$, by applying Theorem~\ref{HypergraphDecomposition} with parameter $r-1$ instead of $r$, we may decompose $G$ into balanced complete $(r-1)$-partite $(r-1)$-graphs $K_{s_1:r-1},\ldots,K_{s_m:r-1}$, with $s_i\le(\log n)^{1/(r-2)}$ and $\sum_{i=1}^m s_i^{r-2}\le c_1n^{r-1}/(\log n)^{1/(r-2)}$, where $c_1=c_1(r)=c'_1(r-1)$. Then we trivially deduce the following two facts.
\begin{itemize}
\item From the second inequality, we have $m\le c_1n^{r-1}/(\log n)^{1/(r-2)}$.

\item Using the fact that these copies of $K_{s_i:r-1}$ are edge disjoint, we have $$\sum_{i=1}^ms_i^{r-1}\le \binom{n}{r-1}<n^{r-1}.$$
\end{itemize}

Therefore, to construct an edge-colored $G$, we need to first choose a sequence of positive integers $(m,s_1,\ldots,s_m)$ such that $m\le c_1n^{r-1}/(\log n)^{1/(r-2)}$, and $s_i\le (\log n)^{1/(r-2)}$ for all $i$. More formally, let 
$$S_{n,r}=\{(m,s_1,s_2,\ldots,s_m):m\le c_1n^{r-1}/(\log n)^{1/(r-2)},1\le s_i\le (\log n)^{1/(r-2)}, 1 \le i \le m\}.$$ 
Then
\begin{align*}
|S_{n,r}|&\le \frac{c_1n^{r-1}}{(\log n)^{1/(r-2)}} \left((\log n)^{1/(r-2)}\right)^{\frac{c_1n^{r-1}}{(\log n)^{1/(r-2)}}}\\
&= 2^{\log\left(\frac{c_1n^{r-1}}{(\log n)^{1/(r-2)}}\right)+\frac{c_1n^{r-1}\log(\log n)^{1/(r-2)}}{(\log n)^{1/(r-2)}}}\\
&\le2^{c_1n^{r-1}}.
\end{align*}

Then, we sequentially construct edge-colored $K_{s_i:r-1}$ for each $i\in[m]$. To make sure $G^*$ is $C_{k}$-free, $K_{s_i:r-1}^*$ has to be made $C_{k}$-free in the first place.  Applying Theorem~\ref{ColorKs:r-1}, we get the following upper bounds. For $r=3$ and even $k$, 
\begin{align*}
g_3(n,k)&\le\sum_{(m,s_1,\ldots,s_m)\in S_{n,3}}\prod_{i=1}^m f_3(n,k,s_i)\\
&\le\sum_{(m,s_1,\ldots,s_m)\in S_{n,3}}\prod_{i=1}^m 2^{(5/2)ks_i\log n+4s_i^2\log D}\\
&\le\sum_{(m,s_1,\ldots,s_m)\in S_{n,3}}2^{\sum_{i=1}^m (5/2)ks_i\log n+4s_i^2\log D}\\
&\le\sum_{(m,s_1,\ldots,s_m)\in S_{n,3}} 2^{(5/2)k\log n(c_1n^2/\log n)+4n^2\log D}\\
&\le 2^{c_1n^2}\cdot2^{(5/2)kc_1n^2+4n^2\log D}\\
&\le 2^{(3kc_1+4\log D)n^2}.
\end{align*}
For $r>3$, and $(m,s_1,\ldots,s_m)\in S_{n,r}$, the number of ways to construct these copies of $K_{s_i:r-1}$ is at most
\begin{align*}
\prod_{i=1}^m f_r(n,k,s_i)&\le\prod_{i=1}^m2^{(c_2+2r)s_i^{r-2}\log n+s_i^{r-1}(\log(c_2+r-1)+(r-2)\log s_i)} \\
&=2^{\sum_{i=1}^m(c_2+2r)s_i^{r-2}\log n+s_i^{r-1}(\log(c_2+r-1)+(r-2)\log s_i)}\\
&\le2^{(c_2+2r)\frac{c_1n^{r-1}}{(\log n)^{1/(r-2)}}\log n+n^{r-1}(\log(c_2+r-1)+(r-2)\log (\log n)^{1/(r-2)})}\\
&\le 2^{(c_2+2r)c_1n^{r-1}(\log n)^{(r-3)/(r-2)}+n^{r-1}(\log(c_2+r-1)+(r-2)\log (\log n)^{1/(r-2)})}\\
&\le 2^{(3/2)(c_2+2r)c_1n^{r-1}(\log n)^{(r-3)/(r-2)}}.
\end{align*}
Note that this is the only place in the proof where we use
$s_i\le (\log n)^{1/(r-2)}$. 

Therefore, 
\begin{align*}
g_r(n,k)&\le \sum_{(m,s_1,\ldots,s_m)\in S_{n,r}} \prod_{i=1}^m f_r(n,k,s_i)\\
&\le2^{c_1n^{r-1}}\cdot 2^{(3/2)(c_2+2r)c_1n^{r-1}(\log n)^{(r-3)/(r-2)}}\\
& \le 2^{2(c_2+2r)c_1n^{r-1}(\log n)^{(r-3)/(r-2)}},
\end{align*}
and the proof is complete.
\end{proof}

\subsection{Finding a cycle if codegrees are high}
A crucial statement that we use in our proof is that any $r$-graph such that every sub-edge has high codegree contains rich structures, including cycles. 
This was explicitly proved in~\cite{KMV15} and we reproduce the proof here for completeness.

\begin{lemma} {\bf (Lemma 3.2 in~\cite{KMV15})} \label{subedge}
For $r,k\ge 3$, if all sub-edges of an $r$-graph $H$ have codegree greater than $rk$, then $C_k\subset H$.
\end{lemma}
\begin{proof}
Let $F=\partial^{r-2}H$ be  the (2-)graph that consists of pairs that are contained in some edge of $H$. Note that each edge of $H$ induces a $K_r$ in $F$, so all edges of $F$ are contained in some triangle ($C_3$). Furthermore, since all sub-edges of $H$ have codegree greater than $rk$, each edge of $F$ is in more than $rk$ triangles. We will first find a $k$-cycle in $F$ as follows. Starting with a triangle $C_3$, for $i=3,\ldots,k-1$ pick an edge $e\in C_{i}$, form $C_{i+1}$ by replacing $e$ by the other two edges of one of at least $rk-i+2$ triangles containing $e$ and excluding other vertices of $C_i$.

Next, let $C_k\subset F$ be a $k$-cycle with edges $f_1,\ldots,f_k$. Find in $H$ a subgraph $C=\{e_i: e_i=f_i\cup g_i, i\in [k]\}$ such that $V(C)=\cup_{i=1}^ke_i$ is of maximum size. Suppose $C$ is not a $k$-cycle in $H$. Then there are
distinct $i, j$ such that $g_i\cap g_j \neq\emptyset$. Pick $v\in g_i\cap g_j$ and consider the sub-edge $e_i \setminus\{v\}=f_i\cup g_i\setminus\{v\}$. The codegree $d_H(e_i \setminus\{v\})>rk$ by assumption. On the other hand, $|V(C)|<rk$ since $C$ is not a $k$-cycle, so there exists a vertex $v'\in N_H(e_i \setminus\{v\})\setminus V(C)$. Replacing $e_i$ by $e_i \setminus\{v\}\cup \{v'\}$, we obtain a $C'$ with a larger vertex set, a contradiction. So $H$ contains a $C_k$.
\end{proof}

\subsection{Proof of Theorem~\ref{Main}}
Now we have all the ingredients to complete the proof of our main result.
\medskip

{\bf Proof of Theorem~\ref{Main}.}
Starting with any $r$-graph $H$ on $[n]$ with $C_{k}\not\subset H$, we claim that there exists a sub-edge with codegree at most $rk$. Indeed, otherwise all sub-edges of $H$ will have codegree more than $rk$, and then by Lemma~\ref{subedge} we obtain a $C_{k}\subset H$. Let $e'$ be the sub-edge of $H$ with $0< d_{H}(e')\le rk$ such that it has smallest lexicographic order among all such sub-edges. Delete all edges of $H$ containing $e'$ from $H$ (i.e. delete $\{e\in H: e'\subset e\}$). Repeat this process of ``searching and deleting" in the remaining $r$-graph until there are no such sub-edges. We claim that the remaining $r$-graph must have no edges at all. Indeed, otherwise we get a nonempty subgraph all of whose sub-edges have codegree greater than $rk$, and again by Lemma~\ref{subedge}, we obtain a $C_{k}\subset H$.

Given  any $C_{k}$-free $r$-graph $H$ on $[n]$, the algorithm above sequentially decomposes $H$ into a collection of sets of at most $rk$ edges who share a sub-edge (an $(r-1)$-set) in common. We regard the collection of these $(r-1)$-sets as an $(r-1)$-graph $G$. Moreover, for each edge $e \in G$, let $N_e$ be the set of vertices $v \in V(H)$ such that $e \cup \{v\}$ is an edge of $H$ at the time $e$ was chosen. So $N_e\in \binom{[n]\setminus e}{\le rk}$, for all $e\in G$. Thus, we get a map 
$$\phi: \Forb_r(n,C_{k}) \longrightarrow \left\{(G,N_G):G\subset \binom{[n]}{r-1}, N_G=\left(N_e\in \binom{[n]\setminus e}{\le rk}:e\in G\right)\right\}.$$
We observe that $\phi$ is injective. Indeed,  
$$\phi^{-1}((G,N_G))=H(N_G)=\{e\cup\{z_e\}:e\in G, z_e\in N_e\},$$ therefore
$|\Forb_r(n,C_{k})|=|\phi(\Forb_r(n,C_{k}))|$. Let $P=\phi(\Forb_r(n,C_{k}))$ which is the set of all pairs $(G,N_G)$ such that $H(N_G)$ is $C_{k}$-free. Next we describe our strategy for upper bounding $|P|$.

For each pair $(G,N_G)\in P$ and $e \in G$, we pick exactly one $z_e^1\in N_e$. Thus we get a pair $(G_1,N_{G_1})$, where $G_1=G$, and $N_{G_1}=(z_e^1:e\in G_1)$. Then, delete $z_e^1$ from each $N_e$, let $G_2=\{e\in G_1: N_e\setminus \{z_e^1\}\neq\emptyset\}$ and pick $z_e^2 \in N_e\setminus \{z_e^1\}$ to get the pair $(G_2, N_{G_2})$. For $2\le i < rk$, we repeat this process for $G_i$ to obtain $G_{i+1}$. Since each $N_{G_i}$ contains only singletons, the pair $(G_i,N_{G_i})$ can be regarded as an edge-colored $(r-1)$-graph. Note that we may get some empty $G_i$s. This gives us a map
$$\psi:P\longrightarrow \left\{(G_1, \ldots, G_{rk}): G_i\subset\binom{[n]}{r-1} \text{ is edge-colored for all $i \in [rk]$}\right\}.$$
Moreover, it is almost trivial to observe that $\psi$ is injective, since if 
$y\ne y'$, then either the underlying $(r-1)$-graphs of $y$ and $y'$ differ, or the $(r-1)$-graphs are the same but the color sets differ. In both cases one can easily see that $\psi(y)\ne \psi(y')$. Again, we let $Q=\psi(P)$.

By Lemma~\ref{ColorHyperG}, we have
\begin{align*}
|\Forb_r(n, C_{k})|=|P|=|Q|&\le\prod_{i=1}^{rk}g_r(n,k)\\
&\le
\begin{cases}
\prod_{i=1}^{3k}2^{(3kc_1+4\log D)n^2} & \text{ if }r=3, k \text{ is even}\\
\prod_{i=1}^{rk}2^{2(c_2+2r)c_1n^{r-1}(\log n)^{(r-3)/(r-2)}} &\text{ if } r>3
\end{cases}\\
&\le
\begin{cases}
2^{\sum_{i=1}^{3k}(3kc_1+4\log D)n^2}& \hskip0.05in \text{ if }r=3, k \text{ is even}\\
2^{\sum_{i=1}^{rk}2(c_2+2r)c_1n^{r-1}(\log n)^{(r-3)/(r-2)}}& \hskip0.05in \text{ if } r>3
\end{cases}\\
&\le
\begin{cases}
2^{(9k^2c_1+12k\log D)n^2}&\hskip0.27in \text{ if }r=3, k \text{ is even}\\
2^{2rk(c_2+2r)c_1n^{r-1}(\log n)^{(r-3)/(r-2)}}&\hskip0.27in\text{ if } r>3
\end{cases}\\
&=2^{c\, n^{r-1}(\log n)^{(r-3)/(r-2)}}.
\end{align*}
where $c=9k^2c_1+12k\log D$ if $r=3$ and $k$ is even, and $c=2rk(c_2+2r)c_1$ if $r>3$. The constants $c_1=c_1(r), c_2=c_2(r,k)$ and $D=D_{k/2}$ are from Theorem~\ref{HypergraphDecomposition} and Theorem~\ref{ColorKs:r-1}, respectively. \qed

\section{Proof of Theorem~\ref{ColorKs:r-1} for $r>3$}
In the remaining part of the paper, we prove Theorem~\ref{ColorKs:r-1}. The cases $r=3$ and $r>3$ have quite different proofs. In this short section we prove the case $r>3$. 

We need one more  ingredient to prove Theorem~\ref{ColorKs:r-1}, namely a partite version of the extremal result for $C_{2l}$. This is a corollary of the main result of~\cite{KMV15} although it can also be proved directly by analyzing the shadow with a much better bound.

\begin{lemma}\label{TuranC2l}
Let $r\ge 3$, $k\ge 3$ and $G$ be an $r$-partite $r$-graph on vertex sets $\bigsqcup_{i=1}^rV_i$ with $|V_i|=s,$ for all $i$. There exists  $c'_2=c'_2(r,k)$, such that if $|G|>c'_2s^{r-1}$ then $G$ contains a cycle of length $k$.
\end{lemma}
\begin{proof}
By Theorem 1.1 of~\cite{KMV15}, we have 
$$\ex_r(n,C_{k})\sim\binom{n}{r}-\binom{n-\lfloor(k-1)/2\rfloor}{r}\sim\frac{\lfloor(k-1)/2\rfloor}{(r-1)!}n^{r-1}.$$
So, we may take $c'_2$ large enough such that
$$c'_2s^{r-1}=\frac{c'_2}{r^{r-1}}n^{r-1}\ge \frac{\lfloor(k-1)/2\rfloor}{(r-1)!}n^{r-1},$$
to guarantee the existence of a copy of $C_{k}$ in $G$. We remark that $c'_2=\frac{\lfloor k/2\rfloor r^{r-1}}{(r-1)!}$ suffices.
\end{proof}

{\bf Proof of Theorem~\ref{ColorKs:r-1} for $r>3$.}
Let $G=K_{s:r-1}$ with $V(G)\subset[n]$ and $C_{k} \not \subset G^*$. For any edge-coloring $N_G=(z_e:e\in G)$ of $G$, let $Z=\{z_e:e\in G\}\subset [n]$ be the set of all its colors. We first argue that  $|Z|<(c_2+r-1)s^{r-2}$, where $c_2=c_2(r,k)=c'_2(r-1,k)=\lfloor k/2\rfloor(r-1)^{r-2}/(r-2)!$, the constant from Lemma~\ref{TuranC2l}. Indeed, if $|Z|\ge(c_2+r-1)s^{r-2}$, then $|Z\setminus V(G)|\ge (c_2+r-1)s^{r-2}-s(r-1)>(c_2+r-1-(r-1))s^{r-2}=c_2s^{r-2}$. For each color  $v \in Z\setminus V(G)$ pick an edge in $G$ with color $v$.  We get a subgraph $G' \subset G$ that is strongly rainbow with $|G'|=|Z\setminus V(G)|>c_2s^{r-2}$. By Lemma~\ref{TuranC2l}, we find an $(r-1)$-uniform  $C_{k}$ in $G$ that is strongly rainbow, which contradicts the fact that $C_{k} \not \subset G^*$.

We now count the number of edge-colored $K_{s:r-1}$ as follows: first choose $s(r-1)$ vertices from $[n]$ as the vertex set, then choose at most $(c_2+r-1)s^{r-2}$ colors, finally color each edge of the $K_{s:r-1}$. As $|K_{s:r-1}|=s^{r-1}$, this yields
\begin{align*}
f_r(n,k,s)&\le n^{s(r-1)+(c_2+r-1)s^{r-2}}((c_2+r-1)s^{r-2})^{s^{r-1}}\\
&=2^{(s(r-1)+(c_2+r-1)s^{r-2})\log n+s^{r-1}(\log(c_2+r-1)+(r-2)\log s)}\\
&\le 2^{(c_2+2r)s^{r-2}\log n+s^{r-1}(\log(c_2+r-1)+(r-2)\log s)}. \qquad \qquad \qed
\end{align*}

\section{Proof of Theorem~\ref{ColorKs:r-1} for $r=3$ and even $k$}

The rest of the paper is devoted to the proof of Theorem~\ref{ColorKs:r-1} for $r=3$ and even $k$. For simplicity of presentation, we write $k=2l$ where $l\ge 2$. Our two main tools are the following lemmas about edge-coloring bipartite graphs.

\begin{lemma}\label{NumberColorKst}
Let $l\ge 2, s,t\ge 1$, $G=K_{s,t}$ be an edge-colored complete bipartite graph with $V(G)\subset [n]$ and $Z=\{z_e:e\in G\}\subset [n]$ be the set of all colors. If $G$ contains no strongly rainbow colored $C_{2l}$, then $|Z|< 2l(s+t)$.
\end{lemma}
\smallskip

\begin{lemma}\label{ColorKst}
For each $l\ge 2$, there exists a constant $D=D_{l}>0$, such that the following holds. Let $s,t\ge 1$, $G=K_{s,t}$ be a complete bipartite graph with vertex set $V(G)\subset [n]$, and $Z\subset [n]$ be a set of colors. Then the number of ways to edge-color $G$ with $Z$ such that the extension $G^*$ contains no $C_{2l}$, is at most $D^{(s+t)^2}$.
\end{lemma}

The proofs of these  lemmas  require several new ideas which will be presented in the rest of the paper. 
Here we quickly show that they imply Theorem~\ref{ColorKs:r-1} for $r=3$ and even $k$.
\medskip

{\bf Proof of Theorem~\ref{ColorKs:r-1} for $r=3$ and $k=2l$.}
Recall that $r \ge 3$, $l \ge 2$, and that $f_3(n,2l,s)$ is the  number of edge-colored copies of $K_{s,s}$
whose vertex set lies in $[n]$ and whose (3-uniform) extension is $C_{2l}$-free. To obtain such a copy of $K_{s,s}$, we first choose from $[n]$ its $2s$ vertices, then its at most $4ls$ colors by Lemma~\ref{NumberColorKst} and finally we color this $K_{s,s}$ by Lemma~\ref{ColorKst}.
This yields
$$f_3(n,2l,s)\le n^{2s+4ls}D^{(2s)^2}\le2^{5ls\log n+4s^2\log D}=2^{(5/2)ks\log n+4s^2\log D},$$ 
where the last inequality holds since $l \ge 2$. Note that $D=D_{l}=D_{k/2}$ is the desired constant.
\qed

\section{Proof of Lemma~\ref{NumberColorKst}}

In this section we prove Lemma~\ref{NumberColorKst}. Our main tool is an extremal result about cycles modulo $h$ in a graph. This problem has a long history, beginning with a Conjecture of Burr and Erd\H{o}s that was solved by Bollob\'as~\cite{B77} in 1976 via the following result: for each integer $m$ and odd positive integer $h$, every graph $G$ with minimum degree $\delta(G)\ge 2((h + 1)^h -1)/h$ contains a cycle of length congruent to $m$ modulo $h$. The lower bound on $\delta(G)$ was improved by Thomassen~\cite{T83} who also generalized it to the case with all integers $h$. It was conjectured by Thomassen that graphs with minimum degree at least $h+1$ contain a cycle of length $2m$ modulo $h$, for any $h, m\ge 1$. The conjecture received new attention recently. In particular Liu-Ma~\cite{LM15} settled the case when $h$ is even, and Diwan~\cite{D10} proved it for $m=2$. To date, Sudakov and Verstra\"ete~\cite{SV15} hold the best known bound for the general case on this problem.

We need the very special case $m=1$ of Thomassen's conjecture and in order to be self contained, we give a proof below. The idea behind this proof can be found in Diwan~\cite{D10}.

\begin{lemma}\label{CycleLength2modk}
If $G$ is an $n$-vertex graph with at least $(h+1)n$ edges, then $G$ contains a cycle of length $2$ modulo $h$.
\end{lemma}
\begin{proof}
By removing vertices of degree at most $h$, we may assume that $G$ has minimum degree at least $h+1$. Let $P$ be a longest path in $G$. Assume that $P$ is of length $l$. Let $V(P)=\{x_0,x_1,\ldots,x_l\}$, where $x_0, x_l$ are the two end-vertices of $P$, and $x_{i-1}x_i\in P$ for all $i\in [l]$. Then we observe that $N(x_0)\subset V(P)$. Otherwise we can extend $P$ to a longer path by $x_0y$ with some vertex $y\in N(x_0)\setminus V(P)$. So $N(x_0)=\{x_i:i\in I\}$ for some $I\subset [l]$. Note that $1\in I$, $|I|=|N(x_0)|\ge \delta(G)\ge h+1$, and the distance $\text{dist}_P(x_i,x_j)=|j-i|$. Consider the set $J=\{i-1:i\in I, i\neq 1\}$, note that this is the set of all distances $\text{dist}_P(x_1,x_i)$ with $x_i\in N(x_0)$ and $i\neq 1$. Clearly, $|J|=|I|-1\ge h$. If there exists $i-1 \in J$ with $i-1\equiv 0$ (mod $h$), then we are done, since the sub-path of $P$ from $x_1$ to $x_i$ together with $x_0x_1$, $x_0x_i$ form a cycle of length $2$ modulo $h$. So none of the numbers in $J$ are multiples of $h$. By the pigeonhole principle, there are at least two elements $i-1,j-1\in J$ such that $i-1\equiv j-1$ (mod $h$), thus $\text{dist}_{P}(x_i,x_j)=|j-i|\equiv 0$ (mod $h$). Again, we can find a cycle of length 2 modulo  $h$ by taking the sub-path of $P$ connecting $x_i, x_j$ and edges $x_0x_i, x_0x_j$. 
\end{proof}

\begin{lemma}\label{C2modimpliesC2l}
Let integers $l \ge 2$, $s,t\ge 1$, $G=K_{s,t}$ with $V(G)\subset [n]$ be edge-colored. If $G$ contains a strongly rainbow colored cycle of length $2$ {\rm(mod  $2l-2$)}, then $G$ contains a strongly rainbow colored $C_{2l}$.
\end{lemma}
\begin{proof}
Let us assume that $C$ is the shortest strongly rainbow colored cycle of length $2$ modulo $2l-2$ in $G$. Then $C$ has at least $2l$ edges. We claim that $C$ is a $C_{2l}$. Suppose not, let $e$ be a chord of $C$ (such a chord exists as $G$ is complete bipartite), such that $C$ is cut up into two paths $P_1$ and $P_2$ by the two endpoints of $e$, and $|P_1|=2l-1$.  Let $Z_1, Z_2$ be the set of their colors respectively. If the color $z_e\notin Z_1\cup V(P_1)\setminus e$, then $P_1\cup e$ is a strongly rainbow colored cycle of length $2l$, a contradiction. Therefore $z_e\in Z_1\cup V(P_1)\setminus e$,  but then $z_e\notin Z_2\cup V(P_2)\setminus e$, yielding a shorter strongly rainbow colored cycle $P_2\cup e$ of length $2$ modulo $2l-2$, a contradiction.
\end{proof}

We now have all the necessary ingredients to prove Lemma~\ref{NumberColorKst}.

\medskip

{\bf Proof of Lemma~\ref{NumberColorKst}.}
Suppose that $|Z|\ge 2l(s+t)$. Then $|Z\setminus V(G)|\ge (2l-1)(s+t)$. For each color $v$ in  $Z\setminus V(G)$, pick an edge $e$ of $G$ with color $v$. We obtain a strongly rainbow colored subgraph $G'$ of $G$ with at least $(2l-1)(s+t)$ edges.
Lemma~\ref{CycleLength2modk} guarantees the existence of a rainbow colored cycle of length $2$ modulo $2l-2$ in $G'$. By construction, this cycle is strongly rainbow. Lemma~\ref{C2modimpliesC2l} then implies that there is a strongly rainbow colored $C_{2l}$ in $G$. \qed

\section{Proof of Lemma~\ref{ColorKst}}

Our proof of Lemma~\ref{ColorKst} is inspired by the methods developed in~\cite{KMV15}. The main idea is to use the bipartite canonical Ramsey theorem. In order to use this approach we  need to develop some new quantitative estimates for an asymmetric version of the bipartite canonical Ramsey theorem. 

\subsection{Canonical Ramsey theory}
In this section we state and prove the main result in Ramsey theory that we will use to prove Lemma~\ref{ColorKst}. We are interested in counting the number of edge-colorings of a bipartite graph, such that the (3-uniform) extension contains no copy of $C_{2l}$. The canonical Ramsey theorem allows us to find nice colored structures that are easier to work with. However, the quantitative aspects are important for our application and consequently we need to prove various bounds for bipartite canonical Ramsey numbers. We begin with some definitions.

Let $G$ be a bipartite graph on vertex set with bipartition $X\sqcup Y$. For any subsets $X'\subset X$, $Y'\subset Y$, let $E_G(X',Y')=G[X'\sqcup Y']=\{xy\in G: x\in X', y\in Y'\}$, and $e_G(X',Y')=|E_G(X',Y')|$. If $X'$ contains a single vertex $x$, then $E_G(\{x\},Y')$ will be simply written as $E_G(x,Y')$. The subscript $G$ may be omitted if it is obvious from context.
\smallskip

\begin{definition}
Let $G$ be an edge-colored bipartite graph with $V(G)=X\sqcup Y$.
\begin{itemize}
\item $G$ is  \emph{monochromatic} if all edges in $E(X,Y)$ are colored by the same color. 
\item $G$ is  \emph{weakly $X$-canonical} if $E(x,Y)$ is monochromatic for each $x\in X$. 
\item $G$ is \emph{$X$-canonical} if it is \emph{weakly $X$-canonical}
and for all distinct $x,x'\in X$ the colors used on $E(x,Y)$ and $E(x',Y)$ are all different.
\end{itemize}
 In all these cases, the color $z_x$ of the edges in $E(x,Y)$ is called a \emph{canonical color}.
\end{definition}
\smallskip

\begin{lemma}\label{WeaklyCanonical}
Let $G=K_{a,b}$ be an edge-colored complete bipartite graph  with bipartition $A\sqcup B$, with $|A|=a, |B|=b$. If $G$ is weakly $A$-canonical, then there exists a subset $A'\subset A$ with $|A'|=\sqrt{a}$ such that $G[A'\sqcup B]=K_{\sqrt{a},b}$ is $A'$-canonical or monochromatic.
\end{lemma}
\begin{proof}
Take a maximal subset $A'$ of $A$ such that the coloring on $E(A',B)$ is $A'$-canonical. If $|A'|\ge \sqrt{a}$, then we are done. So, we may assume that $|A'|<\sqrt{a}$. By maximality of $A'$, there are less then $\sqrt{a}$ canonical colors. By the pigeonhole principle, there are at least $|A|/|A'|\ge a/\sqrt{a} = \sqrt{a}$ vertices of $A$ sharing the same canonical color, which gives a monochromatic $K_{\sqrt{a},b}$.
\end{proof}

Our next lemma guarantees that in an ``almost" rainbow colored complete bipartite graph, there exists a rainbow complete bipartite graph.
\begin{lemma}\label{FindRainbowKcc}
For any integer $c\ge 2$, and $p>c^4$, if $G=K_{p,p}$ is an edge-colored complete bipartite graph, in which each color class is a matching, then $G$ contains a rainbow colored $K_{c,c}$.
\end{lemma}
\begin{proof}
Let $A\sqcup B$ be the vertex set of $G$. Pick two $c$-sets $X,Y$ from $A$ and $B$ respectively at random with uniform probability. For any pair of monochromatic edges $e, e'$, the probability that they both appear in the induced subgraph $E(X,Y)$ is
$$\left(\frac{\binom{p-2}{c-2}}{\binom{p}{c}}\right)^2=\left(\frac{c(c-1)}{p(p-1)}\right)^2.$$
On the other hand, the total number of pairs of monochromatic edges is at most $p^3/2$, since every color class is a  matching.  Therefore the union bound shows that, when $p>c^4$, the probability that there exists a monochromatic pair of edges in $E(X,Y)$ is at most
$$\frac{p^3}{2}\left(\frac{c(c-1)}{p(p-1)}\right)^2 = \frac{pc^4}{2(p-1)^2}<1.$$
Consequently,  there exists a choice of $X$ and $Y$ such that the $E(X,Y)$ contains no pair of monochromatic edges. Such an $E(X,Y)$ is a rainbow colored $K_{c,c}$.
\end{proof}

Now we are ready to prove the main result of this section which is a quantitative version of a result from~\cite{KMV15II}. Note that the edge-coloring in this result uses an arbitrary set of colors. Since the conclusion is about ``rainbow" instead of ``strongly rainbow", it is not essential to have the set of colors disjoint from the vertex set of the graph.

\begin{theorem}[Asymmetric bipartite canonical Ramsey theorem]\label{CanonialRamsey}
For any integer $l\ge 2$, there exists real numbers $\epsilon=\epsilon(l)>0, s_0=s_0(l)>0$, such that if $G=K_{s,t}$ is an edge-colored complete bipartite graph on vertex set $X\sqcup Y$ with $|X|=s, |Y|=t$ with $s>s_0$ and $s/\log s < t\le s$, then one of the following holds:
\begin{itemize}
\item $G$ contains a rainbow colored $K_{4l,4l}$,
\item $G$ contains a $K_{q,2l}$ on vertex set $Q\sqcup R$, with $|Q|=q, |R|=2l$ that is $Q$-canonical, where $q=s^{\epsilon}$,
\item $G$ contains a monochromatic $K_{q,2l}$ on vertex set $Q\sqcup R$, with $|Q|=q, |R|=2l$, where $q=s^\epsilon$.
\end{itemize}
Note that in the last two cases, it could be $Q\subset X, R\subset Y$ or the other way around.
\end{theorem}
\begin{proof}
We will show that $\epsilon=1/18l$. First, fix a subset $Y'$ of $Y$ with $|Y'|=t^{1/4l}$ and let
$$W=\left\{x\in X: \text{there exists a $Y''\in \binom{Y'}{2l}$ such that $E_G(x,Y'')$ is monochromatic}\right\}.$$

If $|W|>s/2l$, then the number of $Y''\in \binom{Y'}{2l}$ such that $E_G(x,Y'')$ is monochromatic for some $x$ (with repetition) is greater than $s/2l$. On the other hand, $|\binom{Y'}{2l}|<|Y'|^{2l}=\sqrt{t}$. By the pigeonhole principle, there exists a $Y''\in \binom{Y'}{2l}$ such that at least
$$\frac{s}{2l\sqrt{t}}\ge \frac{s}{2l\sqrt{s}}\ge s^{1/3}$$
vertices $x$ have the property that $E_G(x,Y'')$ is monochromatic. Let $Q_1$ be a set of $s^{1/3}$ such $x$. Then we obtain a weakly $Q_1$-canonical $K_{s^{1/3},2l}$ on $Q_1\sqcup Y''$ which, by Lemma~\ref{WeaklyCanonical}, contains a canonical or monochromatic $K_{s^{1/6},2l}$. Since $\epsilon<1/6$, this contains a $K_{s^{\epsilon}, 2l}$ as desired.

We may now assume that $|W|\le s/2l$. By definition of $W$ and the pigeonhole principle,  $E_G(x,Y')$ contains at least $|Y'|/2l$ (distinct) colors for every  $x\in X\setminus W$. Hence, for each $x\in X\setminus W$ we can take $|Y'|/2l$ distinctly colored edges from $E(x,Y')$ to obtain a subgraph $G'$ of $G$ on $(X\setminus W)\sqcup Y'$ with $|X\setminus W||Y'|/2l$ edges.

Pick a subset $X'\subset X\setminus W$ with $|X'|=s^{1/16l^2}$ and $e_{G'}(X',Y')\ge |X'||Y'|/2l$. This is possible by an easy averaging argument. Let 
$$Z=\left\{y\in Y': \text{there exists an $X''\in \binom{X'}{2l}$ such that $E_{G'}(X'',y)$ is monochromatic}\right\}.$$

If $|Z|>|Y'|/20l$, then the number of $X''\in \binom{X'}{2l}$ such that $E_{G'}(X'',y)$ is monochromatic for some $y$ (with repetition) is greater than $|Y'|/20l$. On the other hand, $|\binom{X'}{2l}|<|X'|^{2l}=s^{1/8l}$. By the pigeonhole principle, there exists a $X''\in \binom{X'}{2l}$ such that at least
$$\frac{|Y'|}{20ls^{1/8l}}=\frac{t^{1/4l}}{20ls^{1/8l}}\ge \frac{s^{1/4l}}{(\log s)^{1/4l}20ls^{1/8l}}\ge s^{1/9l}=s^{2\epsilon}$$
vertices $y$ have the property that $E_{G'}(X'',y)$ is monochromatic. Let $Q_2$ be a set of $s^{2\epsilon}$ such $y$. We find a weakly $Q_2$-canonical $K_{2l,s^{2\epsilon}}$ on $X''\sqcup Q_2$. Again, by Lemma~\ref{WeaklyCanonical}, a copy of $K_{2l, s^{\epsilon}}$ that is monochromatic or canonical is obtained.

Finally, we may assume that $|Z|\le |Y'|/20l$. Then
\begin{align*}
e_{G'}(X',Y'\setminus Z)&\ge  e_{G'}(X',Y')-|X'||Z|\ge \frac{1}{2l}|X'||Y'|-\frac{1}{20l}|X'||Y'|=\frac{9}{20l}|X'||Y'|\\
&\ge  \frac{9}{20l}|X'||Y'\setminus Z|.
\end{align*}

Since each vertex $y\in Y'\setminus Z$ has the property that  $E_{G'}(X',y)$ sees each color at most $2l-1$ times, for each $y\in Y'\setminus Z$ we may remove all edges from $E_{G'}(X',y)$ with duplicated colors (keep one for each color). We end up getting a bipartite graph $G''$ on $X'\sqcup (Y'\setminus Z)$ with at least $9|X'||Y'\setminus Z|/40l^2$ edges. By the K\H{o}v\'{a}ri-S\'{o}s-Tur\'{a}n theorem~\cite{KST54}, there is a $c>0$ such that $G''$ contains a copy $K$ of $K_{p,p}$ where $p>c\log s$. 
Let  $V(K)=A\sqcup B$. 
For each $x\in A$, the edges set $E(x,B)$  is rainbow colored, and for each $y\in B$, the edge set $E(A,y)$ is rainbow colored.
Therefore each color class in $K$ is a matching. By Lemma~\ref{FindRainbowKcc} and $s>s_0>2^{(4l)^4/c}$, we can find a rainbow colored $K_{4l,4l}$ in $K$ as desired.
\end{proof}

\subsection{The Induction argument for Lemma~\ref{ColorKst}}
We are now ready to prove Lemma~\ref{ColorKst}.  Let us recall the statement. 

{\bf Lemma~\ref{ColorKst}}
\emph{For each $l\ge 2$, there exists a constant $D=D_l>0$, such that the following holds. Let $s,t\ge 1$, $G=K_{s,t}$ be a complete bipartite graph with vertex set $V(G)\subset [n]$, and $Z\subset [n]$ be a set of colors. Then the number of ways to edge-color $G$ with $Z$ such that the extension $G^*$ contains no $C_{2l}$, is at most $D^{(s+t)^2}$.}
\bigskip

{\bf Proof of Lemma~\ref{ColorKst}.}
Let the vertex set of $G$ be $S\sqcup T$ with $|S|=s$ and $|T|=t$. We apply induction on $s+t$. By Lemma~\ref{NumberColorKst},  $|Z|:=\sigma<2l(s+t)$.
The number of ways to color $G$ is at most
$\sigma^{st}$.
 As long as $s+t\le D/2l$, we have
$$\sigma^{st}\le D^{st}\le D^{(s+t)^2}$$
and this concludes the base case(s).

For the induction step, we may henceforth assume $s+t>D/2l$, and the statement holds for all smaller values of $s+t$.  Let us also assume without loss of generality that $t \le s$. 

Next, we deal with the  case $t\le s/\log s$. Let $D>16l^2$. Then $s>(s+t)/2>D/4l>4l$ and the number of ways to color $G$ is at most
$$\sigma^{st}\le (2l(s+t))^{st}\le 2^{\frac{s^2\log (2l(s+t))}{\log s}}\le 2^{\frac{s^2\log (4ls)}{\log s}}\le 2^{2s^2}\le 2^{(s+t)^2\log D}=D^{(s+t)^2}.$$

Therefore, we may assume that $s/\log s<t\le s$, and $s>D/4l>s_0(l)$ so the conditions of Theorem~\ref{CanonialRamsey} hold. Let $N_G=(z_e)_{e\in G}$ be an edge-coloring of $G$ using colors in $Z$. By Theorem~\ref{CanonialRamsey}, such an edge-colored $G$ will contain a subgraph $G'$ that is either
\begin{itemize}
\item a rainbow  colored $K_{4l,4l}$, or
\item a $Q$-canonical $K_{q,2l}$, or
\item a monochromatic $K_{q,2l}$,
\end{itemize} where $|Q|=q=s^{\epsilon}$.

\begin{claim} \label{rainbow} $G'$ cannot be a rainbow colored $K_{4l, 4l}$.
\end{claim}
\begin{proof}[Proof of Claim~\ref{rainbow}]
Suppose for a contradiction that $G'=K_{4l, 4l}$ is rainbow colored  and $Z'$ is the set of colors used on $G'$. Then $|Z'\setminus V(G')|\ge 16l^2-8l$. Pick an edge of each color in $Z'\setminus V(G')$ to obtain a strongly rainbow colored subgraph $G''$ of $G'$ with $|G''|=16l^2-8l\ge (2l-1)8l$. By Lemma~\ref{CycleLength2modk}, $G''$ contains a strongly rainbow colored cycle of length $2\mod 2l-2$. Lemma~\ref{C2modimpliesC2l} now implies the existence of a strongly rainbow colored $C_{2l}$ in $G''$, which forms a linear $C_{2l}$ in $G^*$, a contradiction. \end{proof}
\smallskip

Let $q=s^{\epsilon}$ and $\alpha$ be the number of edge-colorings of $G$ that contain a $Q$-canonical subgraph $G'$ which is a copy of $K_{q,2l}$ and let
$\beta$ be the number of edge-colorings of $G$ that contain a monochromatic subgraph $G'$ which is a copy of $K_{q,2l}$.
 We will prove that both $\alpha$ and $\beta$ are at most $(1/2)D^{(s+t)^2}$ and conclude by Claim~\ref{rainbow} that the total number of colorings is at most $\alpha+\beta\le D^{(s+t)^2}$ as desired.

 Let the vertex set of $G'=K_{q,2l}$ be $Q\sqcup R$, where $Q\in \binom{X}{q}$, $R\in\binom{Y}{2l}$ and 
$\{X,Y\}=\{S,T\}$. Define $a=|X|$ and $b=|Y|$ so $\{a,b\}=\{s,t\}$.

 \subsubsection{ The canonical case}
Our goal is to show that $\alpha \le (1/2)D^{(s+t)^2}$. Recall that for each $x\in Q$, the edges in $E(x,R)$ all have the same color $z_x$ which is called a \emph{canonical color}. Let $Z_c=\{z_x: x\in Q\}$ be the set of all canonical colors. For each edge $xy$ with $x\in Q, y\in Y\setminus(R\cup Z_c)$, a color $z_{xy}\neq z_x$ is called a \emph{free color}.
 We will count the number of colorings of $E(Q,Y)$, and then remove $Q$ to apply the induction hypothesis. For each coloring $N_G$, consider the following partition of $Y\setminus (R\cup Z_c)$  into two parts:
\begin{align*}
Y_0&=\{y\in Y\setminus (R\cup Z_c): \text{$E(y,Q)$ sees at most $11l-1$ free colors}\},\\
Y_1&=\{y\in Y\setminus (R\cup Z_c): \text{$E(y,Q)$ sees at least $11l$ free colors}\}.
\end{align*}

We claim that the length of strongly rainbow colored paths that lie between $Q$ and $Y_1$ is bounded.

\begin{claim}\label{PathLengthII}
If there exists a strongly rainbow colored path $P=P_{2l-2}\subset E(Q,Y_1)$ with both end-vertices $u,v\in Q$, then there exists a $C_{2l}$ in $G^*$.
\end{claim}
\begin{proof}[Proof of Claim~\ref{PathLengthII}]
Clearly, $P$ extends to a linear $P_{2l-2}$ in $G^*$. We may assume both $z_u, z_v\notin V(P^*)$, where $P^*=\{e\cup \{z_e\}:e\in P\}$ is the extension of $P$. Otherwise, suppose w.l.o.g. $z_u\in V(P^*)$, let $y$ be the vertex next to $u$ in $P$, let $S_y$ be of maximum size among sets
$$\{x\in Q: \text{$xy$ all colored by distinct free colors}\}.$$
Since $y\in Y_1$, $|S_y|\ge 11l$. Note that $|V(P^*)|=4l-3$ and $|V(P^*)\cap Y_1|\ge l-1$, we have $|S_y\setminus V(P^*)|\ge 11l-(4l-3-(l-1))\ge 8l$. Since $|V(P^*)|<4l$, $E(y,S_y)$ is rainbow, and $G'$ is $Q$-canonical, there must be at least $4l$ vertices in $S_y\setminus V(P^*)$ whose canonical color is not in $V(P^*)$. Among these $4l$ vertices there is at least one $u'$ with $z_{u'y}\notin V(P^*)$. Replacing $u$ by $u'$, we get a strongly rainbow colored path of length $2l-2$ with $z_u\notin V(P^*)$.

Now, Since $|R|=2l$, we can find a vertex $y\in R$ such that $y\notin \{z_e: e\in P\}$. Further, since both $z_u, z_v\notin V(P^*)$ and $z_u \neq z_v$, the set of edges
$$P^*\cup \{uyz_u,vyz_v\}$$
forms a copy of $C_{2l}$ in $G^*$.
\end{proof}

Thanks to this observation about strongly rainbow paths, we can bound the number of colorings on $E(Q,Y_1)$ as follows. It is convenient to use the following notation.
\smallskip

\begin{definition} Given $X' \subset X$ and $Y' \subset Y$, let $\#E(X',Y')$ be the number of ways to color the edges in $E(X', Y')$. 
\end{definition}
\smallskip

\begin{claim}\label{ColorQY1II}
$\#E(Q,Y_1)\le(2l)^q\cdot (32l^2)^{bq}\cdot (qb)^{2lq}\cdot\sigma^{6lq+8l^2b}.$
\end{claim}

\begin{proof}[Proof of Claim~\ref{ColorQY1II}]
By Claim~\ref{PathLengthII}, according to the length of the longest strongly rainbow colored path starting at a vertex, $Q$ can be partitioned into  $2l-3$ parts $\bigsqcup_{i=1}^{2l-3}Q_{i}$, where
\begin{align*}
Q_i=\{&x\in Q: \text{the longest strongly rainbow colored path}\\ &\text{starting at $x$ and contained in $E(Q,Y_1)$ has length $i$}\}.
\end{align*}
For each $i$, let $q_i=|Q_i|$.  We now bound the number of colorings of the edges in $E(Q_i,Y_1)$.

Firstly, for each $x\in Q_i$, choose an $i$-path $P_x\subset E(Q,Y_1)$ starting at $x$ and color it strongly rainbow. The number of ways to choose and color these paths for all the vertices $x \in Q_i$  is at most
$$((qb)^{\lceil (i+1)/2\rceil}\sigma^i )^{q_i}\le (qb\sigma)^{iq_i}.$$

Fix an $x\in Q_i$.  Partition $Y_1$  into $3$ parts depending on whether  $y$ is on the extension $P_x^*$ of the path starting at $x$, or the color of  $xy$
is on $P_x^*$ or else,  i.e. $Y_1=\bigsqcup_{j=1}^{3}Y_{i,x}^{(j)}$, where
\begin{align*}
Y_{i,x}^{(1)}&=Y_1\cap V(P_x^*),\\
Y_{i,x}^{(2)}&=\{y\in Y_1\setminus Y_{i,x}^{(1)}: z_{xy}\in V(P_x^*)\},\\
Y_{i,x}^{(3)}&=Y_1\setminus(Y_{i,x}^{(1)}\cup Y_{i,x}^{(2)}).
\end{align*}

Depending on the part of $Y_1$ that a vertex $y$ lies in, we can get different restrictions on the coloring of the edges in $E(y, Q_i)$.

\begin{itemize}

\item If $y\in Y_{i,x}^{(1)}$, then $z_{xy}$ has as many as $\sigma$ choices. Note that $|P_x^*|=2i+1$, and $|Y_{i,x}^{(1)}|\le i+\lceil i/2\rceil\le 2i$. This gives  $\#E(x,Y_{i,x}^{(1)})\le \sigma^{2i}$.

\item If $y\in Y_{i,x}^{(2)}$, then $z_{xy}\in V(P_x^*)$, so there are at most $2i+1$ choices for this color and $\#E(x,Y_{i,x}^{(2)}) \le (2i+1)^b$.

\item Lastly, let $|Y_{i,x}^{(3)}|=b_{i,x}$. If $y\in Y_{i,x}^{(3)}$, then $xy$ extends $P_x$ into a strongly rainbow colored path $P_x'=P_x\cup \{xy\}$ of length $i+1$, which forces the edges $x'y$ to be colored by $V({P_x'}^*)$ for each $x'\in Q_i\setminus V({P_x'}^*)$. Otherwise, the path $P_x'\cup \{x'y\}$ is a strongly rainbow colored path of length $i+2$ starting at a vertex $x'\in Q_i$, contradicting  the definition of $Q_i$. Therefore, $z_{x'y}$ has at most $2i+3$ choices if $x'\in Q_i\setminus V({P_x'}^*)$. Putting this together, for each $y \in Y_{i,x}^{(3)}$, we have 
$$\#E(Q_i\setminus V({P_x'}^*),y)\le (2i+3)^{q_i}.$$
Noticing that $|Q_i\cap V({P_x'}^*)|\le i+1+\lceil(i+1)/2\rceil\le 2i+1$, we have 
$$\#E(Q_i,y) \le \#E(Q_i\cap V({P_x'}^*),y)\cdot\#E(Q_i\setminus V({P_x'}^*),y)\le \sigma^{2i+1}(2i+3)^{q_i}.$$

\end{itemize}
Hence the number of ways to  color $E(x,Y_1)\cup E(Q_i,Y_{i,x}^{(3)})$ is at most
\begin{align*}
2^b\cdot\sigma^{2i}\cdot (2i+1)^b\cdot \sigma^{(2i+1)b_{i,x}}(2i+3)^{q_ib_{i,x}}.
\end{align*}
The term $2^b$ arises above since  $Y_{i,x}^{(1)}$ has already been fixed before this step, so we just need to partition $Y_1\setminus Y_{i,x}^{(1)}$ to get $Y_{i,x}^{(2)}$ and $Y_{i,x}^{(3)}$.

Now we remove $x$ from $Q_i$, $Y_{i,x}^{(3)}$ from $Y_1$ and repeat the above steps until we have the entire $E(Q_i,Y_1)$ colored. Note that $\sum_{x\in Q_i}b_{i,x}\le b$, and that $i\le 2l-3$ which implies $2i+3<4l$. We obtain
\begin{align*}
\#E(Q_i,Y_1)&\le(qb\sigma)^{iq_i}\prod_{x\in Q_i}2^b\cdot\sigma^{2i}\cdot (2i+1)^b\cdot \sigma^{(2i+1)b_{i,x}}(2i+3)^{q_ib_{i,x}}\\
&\le(qb\sigma)^{2lq_i}\prod_{x\in Q_i}2^b\cdot\sigma^{4l+4lb_{i,x}}\cdot(4l)^{b+q_ib_{i,x}}\\
&\le(qb\sigma)^{2lq_i}\cdot2^{bq_i}\cdot\sigma^{4lq_i+4lb}\cdot(4l)^{bq_i+bq_i}\\
&=(32l^2)^{bq_i}\cdot (qb)^{2lq_i}\cdot\sigma^{6lq_i+4lb}.
\end{align*}

Because $\sum_{i=1}^{2l-3}q_i=q$, taking the product over $i\in[2l-3]$, we obtain
\begin{align*}
\#E(Q,Y_1) \le (2l-3)^q \prod_{i=1}^{2l-3} \#E(Q_i, Y_1)
&\le(2l-3)^q\prod_{i=1}^{2l-3}(32l^2)^{bq_i}\cdot (qb)^{2lq_i}\cdot\sigma^{6lq_i+4lb}\\
&\le (2l)^q\cdot (32l^2)^{bq}\cdot (qb)^{2lq}\cdot\sigma^{6lq+8l^2b},
\end{align*}
where $(2l-3)^q$ counts the number of partitions of $Q$ into the $Q_i$.
\end{proof}

Since $G'=E(Q,R)$ is $Q$-canonical, 
$$\#E(Q,R)\le \sigma^q.$$
As $|Z_c|\le q$, 
$$\#E(Q,Y\cap Z_c)\le \sigma^{q^2}.$$
By definition of $Y_0$,  
$$\#E(Q,Y_0)\le (\sigma^{11l}(11l+1)^q)^b\le (\sigma^{11l}(12l)^{q})^b.$$

Therefore to color $E(Q,Y)$, we need to first choose the subsets $R$ and $Z_c\cap Y$ of $Y$ and then take a partition to get $Y_0$ and $Y_1$. We color each of  $E(Q,R), E(Q,Y\cap Z_c),E(Q,Y_0)$ and $E(Q,Y_1)$. This gives 
\begin{align*}
\#E(Q,Y)&\le b^{2l}b^q2^b\cdot \#E(Q,R)\cdot \#E(Q,Y\cap Z_c)\cdot \#E(Q,Y_0)\cdot \#E(Q,Y_1)\\
&\le b^{2l}b^q2^b\cdot \sigma^q \cdot  \sigma^{q^2} \cdot (\sigma^{11l}(12l)^{q})^b\cdot[(2l)^q\cdot (32l^2)^{bq}\cdot (qb)^{2lq}\cdot\sigma^{6lq+8l^2b}]\\
&=b^{2l}2^b\cdot(2lb)^q\cdot(384l^3)^{qb}\cdot(qb)^{2lq}\cdot\sigma^{q+q^2+11lb+6lq+8l^2b}.
\end{align*}

Finally, we apply the induction hypothesis to count the number ways to color $G=K_{s,t}$. Recall that $q=s^\epsilon<s/\log s<t\le s$, $\sigma\le 2l(s+t)\le 4ls$. There are two cases.
\smallskip

{\bf $\bullet$ $(X, Y)=(S,T)$ and $(a,b)=(s,t)$}

 Recall that we must first choose $Q\subset X$.
\begin{align*}
\#E(X,Y)&\le s^q\cdot\#E(Q,Y)\cdot\#E(X\setminus Q,Y)\\
&\le s^q \cdot t^{2l} 2^t\cdot(2lt)^q\cdot(384l^3)^{qt}\cdot(qt)^{2lq}\cdot\sigma^{q+q^2+11lt+6lq+8l^2t}\cdot D^{(s+t-q)^2}\\
&\le s^q t^{2l}\cdot 2^t\cdot(2lt)^q\cdot(384l^3)^{qt}\cdot(qt)^{2lq}\cdot(4ls)^{q+q^2+11lt+6lq+8l^2t}\cdot D^{(s+t-q)^2}\\
&\le  s^q t^{2l}\cdot t^q\cdot(2^{1/q}(2l)^{1/t}384l^3)^{qt}\cdot(qt)^{2lq}\cdot(4ls)^{9l^2t}\cdot D^{-2qt}\cdot D^{-2qs+q^2} \cdot D^{(s+t)^2}\\
&\le \left(\frac{2^{1/q}(2l)^{1/t}384l^3(4l)^{9l^2/q}}{D^2}\right)^{qt}\cdot t^{2l}(st)^q(qt)^{2lq}s^{9l^2t}\cdot D^{-qs}\cdot D^{(s+t)^2}\\
&\le\left(\frac{2^{1/q}(2l)^{1/t}384l^3(4l)^{9l^2/q}}{D^2}\right)^{qt}\cdot \frac{q^{2lq}t^{2l+q+2lq}s^{q+9l^2t}}{D^{qs}}\cdot D^{(s+t)^2}\\
&\le  \frac{1}{4}D^{(s+t)^2}.
\end{align*}
To show the last inequality above, it is obvious  that ${2^{1/q}(2l)^{1/t}384l^3(4l)^{9l^2/q}}<{D^2}$ for large enough $D$, so we are left to show that ${4q^{2lq}t^{2l+q+2lq}s^{q+9l^2t}}<{D^{qs}}$ for large $D$. Taking logarithms, we have
\begin{align*}
\log \left({4q^{2lq}t^{2l+q+2lq}s^{q+9l^2t}}\right)&=2+2lq\log q+(2l+q+2lq)\log t+(q+9l^2t)\log s\\
&\le qs\log D.
\end{align*}
This holds for large  $D$ because $qs$ has the dominating growth rate among all the above terms as $q=s^{\epsilon}$ and $s$ is large.
\smallskip

{\bf $\bullet$  $(X,Y)=(T,S)$ and $(a,b)=(t,s)$}
\begin{align*}
\#E(X,Y)&\le t^q\cdot\#E(Q,X)\cdot\#E(Y\setminus Q,X)\\
&\le t^q \cdot s^{2l}2^s\cdot(2ls)^q\cdot(384l^3)^{qs}\cdot(qs)^{2lq}\cdot\sigma^{q+q^2+11ls+6lq+8l^2s}\cdot D^{(s+t-q)^2}\\
&\le t^q s^{2l}\cdot 2^s\cdot(2ls)^q\cdot(384l^3)^{qs}\cdot(qs)^{2lq}\cdot(4ls)^{q+q^2+11ls+6lq+8l^2s}\cdot D^{(s+t-q)^2}\\
&\le  t^q s^{2l}\cdot s^q\cdot(2^{1/q}(2l)^{1/s}384l^3)^{qs}\cdot(qs)^{2lq}\cdot(4ls)^{9l^2s}\cdot D^{-2qs}\cdot D^{-2qt+q^2} \cdot D^{(s+t)^2}\\
&\le \left(\frac{2^{1/q}(2l)^{1/s}384l^3(4l)^{9l^2/q}}{D^2}\right)^{qs}\cdot s^{2l}(st)^q(qs)^{2lq}s^{9l^2s}\cdot D^{-qt}\cdot D^{(s+t)^2}\\
&\le\left(\frac{2^{1/q}(2l)^{1/s}384l^3(4l)^{9l^2/q}}{D^2}\right)^{qs}\cdot \frac{q^{2lq}t^{q}s^{2l+q+2lq+9l^2t}}{D^{qt}}\cdot D^{(s+t)^2}\\
&\le  \frac{1}{4}D^{(s+t)^2}.
\end{align*}
Again, to show the last inequality above, it is clear that ${2^{1/q}(2l)^{1/s}384l^3(4l)^{9l^2/q}}<{D^2}$ for large $D$, so we are left to show that ${4q^{2lq}t^{q}s^{2l+q+2lq+9l^2s}}<{D^{qs}}$ for large  $D$. Taking logarithms, we have
\begin{align*}
\log \left({4q^{2lq}t^{q}s^{2l+q+2lq+9l^2s}}\right)&=2+2lq\log q+q\log t+(2l+q+2lq+9l^2s)\log s\\
&\le qt\log D.
\end{align*}
This holds for large  $D$ because $qt$ has the dominating growth rate among all the above terms.

In summary, the number of colorings of $G$ such that there exists a $G'\subset G$ that is a $Q$-canonical $K_{q,2l}$ is
$$\alpha\le \frac{1}{4}D^{(s+t)^2}+\frac{1}{4}D^{(s+t)^2}=\frac{1}{2}D^{(s+t)^2}.$$
\subsubsection{The monochromatic case}
Our goal is to show that $\beta \le (1/2)D^{(s+t)^2}$. Recall that the vertex set of $G'=K_{q,2l}$ is $Q\sqcup R$, where $Q\in \binom{X}{q}$ and $R\in\binom{Y}{2l}$. The term \emph{canonical color} now refers to the only color $z_c$ that is used to color all edges of $G'$, and $Z_c=\{z_c\}$ still means the set of canonical colors. A {\it free color} is a color that is not $z_c$. As before we will count the number of colorings of $E(Q,Y)$, and then remove $Q$ to apply the induction hypothesis.

Let $Y_1=Y\setminus (R\cup Z_c)$. Similar to Claim~\ref{PathLengthII}, we claim that the length of a strongly rainbow colored path between $Q$ and $Y_1$ is bounded.

\begin{claim}\label{PathLengthIII}
If there exists a strongly rainbow colored path $P=P_{4l-2}\subset E(Q,Y_1)$ with both end-vertices $u,v\in Q$, then there exists a $C_{2l}$ in $G^*$.
\end{claim}
\begin{proof}[Proof of Claim~\ref{PathLengthIII}]
We observe that $z_c$ appears in the path or the color of the path at most once, as $P$ is strongly rainbow. Hence, by the pigeonhole principle, there exists a sub-path $P'$ of length $2l-2$ such that $z_c\notin V({P'}^*)$ and both end-vertices $u,v$ of $P'$ are in $Q$.

Now, Since $|R|=2l$, we can find two vertices $y, y'\in R$ such that $y, y'\notin \{z_e: e\in P'\}$. Thus, the edges
$${P'}^*\cup \{uyz_c,vy'z_c\}$$
yield a copy of $C_{2l}$ in $G^*$.
\end{proof}

Again, we first use this claim to color $E(Q,Y_1)$.

\begin{claim}\label{ColorQY1III}
$\#E(Q,Y_1) \le (4l)^q\cdot (128l^2)^{qb}\cdot (qb)^{4lq}\cdot\sigma^{12lq+32l^2b}$.
\end{claim}
\begin{proof}[Proof of Claim~\ref{ColorQY1III}]
The proof proceeds exactly the same as that of Claim~\ref{ColorQY1II}, except that $Q$ is partitioned into 
$4l-3$ parts $\bigsqcup_{i=1}^{4l-3}Q_{i}$. So in the calculation at the end, we have $i\le 4l-3$ which gives $2i+3<8l$ and
\begin{align*}
\#E(Q_i,Y_1)&\le(qb\sigma)^{iq_i}\prod_{x\in Q_i}2^b\cdot\sigma^{2i}\cdot (2i+1)^b\cdot \sigma^{(2i+1)b_{i,x}}(2i+3)^{q_ib_{i,x}}\\
&\le(qb\sigma)^{4lq_i}\prod_{x\in Q_i}2^b\cdot\sigma^{8l+8lb_{i,x}}\cdot(8l)^{b+q_ib_{i,x}}\\
&\le(qb\sigma)^{4lq_i}\cdot2^{bq_i}\cdot\sigma^{8lq_i+8lb}\cdot(8l)^{bq_i+bq_i}\\
&\le(128l^2)^{bq_i}\cdot (qb)^{4lq_i}\cdot\sigma^{12lq_i+8lb}.
\end{align*}

Again, note that $\sum_{i=1}^{4l-3}q_i=q$. Taking the product over $i\in[4l-3]$, we obtain
\begin{align*}
\#E(Q,Y_1)\le (4l-3)^q\prod_{i=1}^{4l-3}\#E(Q_i, Y_1)&\le(4l-3)^q\prod_{i=1}^{4l-3}(128l^2)^{bq_i}\cdot (qb)^{4lq_i}\cdot\sigma^{12lq_i+8lb}\\
&\le (4l)^q\cdot (128l^2)^{qb}\cdot (qb)^{4lq}\cdot\sigma^{12lq+32l^2b},
\end{align*}
where $(4l-3)^q$ counts the number of partitions of $Q$ into the $Q_i$.
\end{proof}
Similarly, to color $E(Q,Y)$, we need to choose the subsets $R$ and $Y\cap Z_c$, and what remains is $Y_1$. Consequently, 
\begin{align*}
\#E(Q,Y)&\le b^{2l}b\cdot \#E(Q,R)\cdot \#E(Q,Y\cap Z_c)\cdot \#E(Q,Y_1)\\
&\le  b^{2l}b\cdot\sigma\cdot \sigma^{q} \cdot[(4l)^q\cdot (128l^2)^{qb}\cdot (qb)^{4lq}\cdot\sigma^{12lq+32l^2b}]\\
&=b^{2l+1}(4l)^q\cdot(128l^2)^{qb}\cdot(qb)^{4lq}\cdot\sigma^{1+q+12lq+32l^2b}.
\end{align*}

We apply the induction hypothesis to count the number ways to color $G=K_{s,t}$. Recall  that $q=s^\epsilon<s/\log s<t\le s$, $\sigma\le 2l(s+t)\le 4ls$. We split the calculation into two cases.
\smallskip

{\bf $\bullet$ $(X,Y)=(S,T)$ and $(a,b)=(s,t)$} 

Recall that we must first choose $Q\subset X$.
\begin{align*}
\#E(X,Y)&\le s^q\cdot\#E(Q,Y)\cdot\#E(X\setminus Q,Y)\\
&\le s^q \cdot t^{2l+1}(4l)^q\cdot(128l^2)^{qt}\cdot(qt)^{4lq}\cdot\sigma^{1+q+12lq+32l^2t}\cdot D^{(s+t-q)^2}\\
&\le s^q t^{2l+1}\cdot (4l)^q\cdot(128l^2)^{qt}\cdot(qt)^{4lq}\cdot(4ls)^{1+q+12lq+32l^2t}\cdot D^{(s+t-q)^2}\\
&\le  s^q t^{2l+1}\cdot ((4l)^{1/t}128l^2)^{qt}\cdot(qt)^{4lq}\cdot(4ls)^{33l^2t}\cdot D^{-2qt}\cdot D^{-2qs+q^2} \cdot D^{(s+t)^2}\\
&\le \left(\frac{(4l)^{1/t}128l^2(4l)^{33l^2/q}}{D^2}\right)^{qt}\cdot t^{2l+1}s^q(qt)^{4lq}s^{33l^2t}\cdot D^{-qs}\cdot D^{(s+t)^2}\\
&\le\left(\frac{(4l)^{1/t}128l^2(4l)^{33l^2/q}}{D^2}\right)^{qt}\cdot \frac{q^{4lq}t^{2l+1+4lq}s^{q+33l^2t}}{D^{qs}}\cdot D^{(s+t)^2}\\
&\le  \frac{1}{4}D^{(s+t)^2}.
\end{align*}
To show the last inequality above, it is obvious to see that ${(4l)^{1/t}128l^2(4l)^{33l^2/q}}<{D^2}$ for large  $D$, so we are left to show that ${4q^{4lq}t^{2l+1+4lq}s^{q+33l^2t}}<{D^{qs}}$ for large $D$. Taking logarithms, we have
\begin{align*}
\log \left({4q^{4lq}t^{2l+1+4lq}s^{q+33l^2t}}\right)&=2+4lq\log q+(2l+1+4lq)\log t+(q+33l^2t)\log s\\
&\le qs\log D.
\end{align*}
This holds for large  $D$ because $qs$ has dominating growth rate among all the terms above.
\smallskip

{\bf $\bullet$ $(X,Y)=(T,S)$ and $(a,b)=(t,s)$}
\begin{align*}
\#E(X,Y)&\le t^q\cdot\#E(Q,Y)\cdot\#E(X\setminus Q,Y)\\
&\le t^q \cdot s^{2l+1}(4l)^q\cdot(128l^2)^{qs}\cdot(qs)^{4lq}\cdot\sigma^{1+q+12lq+32l^2s}\cdot D^{(s+t-q)^2}\\
&\le t^q s^{2l+1}\cdot (4l)^q\cdot(128l^2)^{qs}\cdot(qs)^{4lq}\cdot(4ls)^{1+q+12lq+32l^2s}\cdot D^{(s+t-q)^2}\\
&\le  t^q s^{2l+1}\cdot ((4l)^{1/s}128l^2)^{qs}\cdot(qs)^{4lq}\cdot(4ls)^{33l^2s}\cdot D^{-2qs}\cdot D^{-2qt+q^2} \cdot D^{(s+t)^2}\\
&\le \left(\frac{(4l)^{1/s}128l^2(4l)^{33l^2/q}}{D^2}\right)^{qs}\cdot s^{2l+1}t^q(qs)^{4lq}s^{33l^2s}\cdot D^{-qt}\cdot D^{(s+t)^2}\\
&\le\left(\frac{(4l)^{1/s}128l^2(4l)^{33l^2/q}}{D^2}\right)^{qs}\cdot \frac{q^{4lq}t^{q}s^{2l+1+4lq+33l^2t}}{D^{qt}}\cdot D^{(s+t)^2}\\
&\le  \frac{1}{4}D^{(s+t)^2}.
\end{align*}
To show the last inequality above, it is obvious to see that ${(4l)^{1/s}128l^2(4l)^{33l^2/q}}<{D^2}$ for large $D$, so we are left to show that ${4q^{4lq}t^{q}s^{2l+1+4lq+33l^2t}}<{D^{qt}}$ for large $D$. Taking logarithms, we have
\begin{align*}
\log \left(4q^{4lq}t^{q}s^{2l+1+4lq+33l^2t}\right)&=2+4lq\log q+q\log t+(2l+1+4lq+33l^2t)\log s\\
&\le qt\log D.
\end{align*}
This holds for large $D$ because $qt$ has the dominating growth rate among all the terms above.

Finally, we have 
$$\beta\le \frac{1}{4}D^{(s+t)^2}+\frac{1}{4}D^{(s+t)^2}=\frac{1}{2}D^{(s+t)^2},$$
and the proof is complete.\qed

\section{Concluding remarks}

$\bullet$ A straightforward corollary of Theorem~\ref{Main} is the very same result for hypergraph paths $P_k$. Indeed, for the upper bound on $\Forb_r(n, P_k)$ one has to just observe that $P_k \subset C_{2\lceil (k+1)/2 \rceil}$, while the lower bound is trivial.

$\bullet$ Although we were unable to use hypergraph containers to prove Theorem~\ref{Main} it would be very interesting to give a new proof using containers. In particular, this would entail proving some supersaturation type results for this problem which may be of independent interest. It would also likely yield some further results in the random setting which we have not addressed.

$\bullet$ The main open problem raised by our work is to solve the analogous question for larger $r$ and for odd cycles (Conjecture~\ref{conj}).

 For $r=3$, our method will not work for odd cycles as it relies on finding a bipartite structure from which it is difficult to extract odd 3-uniform cycles (although this technical hurdle could be overcome to solve the corresponding extremal problem in~\cite{KMV15}).
 
  For larger $r$, our method does not  work because the cost of decomposing a complete $r$-graph into complete $r$-partite subgraphs is too large to remain an error term. More precisely, for $r=3$, we implicitly applied Lemma~\ref{ColorKst}  (in the proof of Lemma~\ref{ColorHyperG}) to reduce the number of ways to color a graph to at most $2^{O(n^2)}$ instead of the trivial $2^{O(n^{2}\log\log n)}$. But for $r>3$ the main term in the calculation turns out to be $2^{O(n^{r-1}(\log n)^{(r-3)/(r-2)})}$ which comes from choosing the  colors for the copies of $K_{s_i:r-1}$ (see Section 3). This cannot be improved due to Theorem~\ref{HypergraphDecomposition} and Lemma~\ref{TuranC2l} each of which gives a bound that is sharp in order of magnitude.
Consequently, even if we proved a version of Lemma~\ref{ColorKst} for $r>3$ (and the tools we have developed should suffice to provide such a proof) this would not improve Theorem~\ref{Main} for $r>3$.

$\bullet$ Another way to generalize the result of Morris-Saxton to hypergraphs is to consider similar enumeration questions when the underlying $r$-graph is linear, meaning that every two edges share at most one vertex. Here the extremal results have recently been proved in~\cite{CGJ} and the formulas are similar to the case of graphs.  The special case of this question for linear triple systems  without a $C_3$ is related to the Ruzsa-Szem\'eredi $(6,3)$ theorem and sets without 3-term arithmetic progressions.
\bigskip

{\bf Acknowledgment.}
We are very grateful to Rob Morris for clarifying some technical parts of the proof in~\cite{MS13} at the early stages of this project.  After those discussions, we realized that the method of hypergraph containers would not apply easily to prove Theorem~\ref{Main} and we therefore developed new ideas. 
We are also grateful to Jozsef Balogh for providing us with some pertinent references, and to Jie Han for pointing out that our proof of Theorem~\ref{Main} applies for $r>3$ and odd $k$.

\end{document}